\documentclass[a4paper,10pt]{article}
\usepackage[utf8]{inputenc}
 \usepackage[T1]{fontenc}
 \usepackage[normalem]{ulem}
 \usepackage[english]{babel}
\usepackage[notref,notcite]{showkeys}

\usepackage{geometry}                
\usepackage{graphicx}
\usepackage{amssymb,amsthm,amsmath}
\usepackage{epstopdf}
\newtheorem{theorem}{Theorem}
\newtheorem{definition}[theorem]{Definition}
\newtheorem{proposition}[theorem]{Proposition}

\newtheorem{lemma}[theorem]{Lemma}
\newtheorem{remark}[theorem]{Remark}

\newcommand{\supp}{\mathrm{supp}\,}

\usepackage{graphicx}
\usepackage{amsmath}
\usepackage{amssymb}
\usepackage{amsthm}
\usepackage{graphicx}
\usepackage{pstricks}
\usepackage{amsfonts}
\usepackage{xr}
\usepackage{color}
\usepackage{graphicx,epsfig}
\usepackage{pdfsync,subfigure}
\usepackage{dsfont}

\newcommand{\R}{\mathbb{R}}
\newcommand{\C}{\mathbb{C}}
\newcommand{\N}{\mathbb{N}}

\newcommand{\Z}{\mathbb{Z}}

\renewcommand{\epsilon}{\varepsilon}
\newcommand{\eps}{\varepsilon}

\newcommand {\Chi} {{\bf \raise 2pt \hbox{$\chi$}} }

\title{Travelling waves for the cane toads equation with bounded traits.}
\date\today                                           
\author{Emeric Bouin\footnote{Corresponding author}\,
\footnote{Ecole Normale Sup\'erieure de Lyon, UMR CNRS 5669 'UMPA',  and INRIA Alpes, project-team NUMED,  46 all\'ee d'Italie, F-69364~Lyon~cedex~07, France. E-mail: \texttt{emeric.bouin@ens-lyon.fr}}\; 
\and
Vincent Calvez\footnote{Ecole Normale Sup\'erieure de Lyon, UMR CNRS 5669 'UMPA', and INRIA Alpes, project-team NUMED, 46 all\'ee d'Italie, F-69364~Lyon~cedex~07, France. E-mail: \texttt{vincent.calvez@ens-lyon.fr}}\;
}

\begin{document}
\maketitle
\begin{abstract}
In this paper, we study propagation in a nonlocal reaction-diffusion-mutation model describing the invasion of cane toads in Australia \cite{Phillips}. The population of toads is structured by a space variable and a phenotypical trait and the space-diffusivity depends on the trait. We use a Schauder topological degree argument for the construction of some travelling wave solutions of the model. The speed $c^*$ of the wave is obtained after solving a suitable spectral problem in the trait variable. An eigenvector arising from this eigenvalue problem gives the flavor of the profile at the edge of the front. {The major difficulty is to obtain uniform $L^\infty$ bounds despite the combination of non local terms and an heterogeneous diffusivity.}
\end{abstract}
\noindent{\bf AMS Class. No:} {35Q92, 45K05, 35C07}

\section{Introduction.}

In this paper, we focus on propagation phenomena in a model for the invasion of cane toads in Australia, proposed in \cite{Benichou}. It is a structured population model with two structural variables, the space $x\in \R^n$ and the motility $\theta \in \Theta$ of the toads. {Here $\Theta:= \left( \theta_{\text{min}}, \theta_{\text{max}} \right)$, with $\theta_{\text{min}} > 0$} denotes the bounded set of traits. {One modeling assumption is that} the space diffusivity depends {only on} $\theta$. The mutations are simply modeled {by} a diffusion process with constant diffusivity $\alpha$ in the variable $\theta$. {Each toad is in local competition} with all other individuals (independently of their trait) for resources. The resulting reaction term is of monostable type. Denoting $n(t,x,\theta)$ the density of toads having trait $\theta \in \Theta$ in position $x \in \R^n$ at time $t \in \R^+$, the model writes:
\begin{equation}\label{eq:main}
\left\{\begin{array}{l}
\partial_t n - \theta \Delta_x n - \alpha \partial_{\theta\theta} n = r n  (1 - \rho)\, , \qquad (t,x,\theta) \in \R^+ \times \R^n \times \Theta, \medskip\\
\partial_\theta n (t,x,\theta_{\min}) = \partial_\theta n (t,x,\theta_{\max}) = 0\, , \qquad (t,x) \in \R^+ \times \R^n.\\
\end{array}
\right.
\end{equation}
with
\begin{equation*}
\forall (t,x) \in \R^+ \times \R^n, \qquad \rho(t,x) = \int_\Theta n(t,x,\theta)\, d\theta.
\end{equation*}
The Neumann boundary conditions ensure {the conservation of individuals through the mutation process}. 

{
The invasion of cane toads has interested several field biologists. The data collected \cite{Shine, Phillips} show that the speed of invasion has always been increasing during the eighty first years of propagation and that younger individuals at the edge of the invasion front have shown significant changes in their morphology compared to older populations. This example of ecological problem among others (see the expansion of bush crickets in Britain \cite{Thomas}) illustrates the necessity of having models able to describe space-trait interactions. Several works have addressed the issue of front invasion in ecology, where the trait is related to dispersal ability \cite{Desvillettes,Champagnat}. It has been postulated that selection of more motile individuals can occur, even if they have no advantage regarding their reproductive rate, due to spatial sorting \cite{Kokko,Ronce,Shine,Simmons}.}

Recently, some models for populations structured simultaneously by phenotypical traits and a space variable have emerged. A similar model to \eqref{eq:main} in a discrete trait setting has been studied by Dockery \textit{et al.} in \cite{Dockery}. Interestingly, they prove that in a bounded space domain {and with a rate of growth $r(x)$ heterogeneous in space}, 
{the only nontrivial Evolutionarily Stable State (ESS) is a population dominated by the slowest diffusing phenotype}. This conclusion is precisely the opposite of what is expected {at the edge of an invading front}. In \cite{Alfaro}, the authors study propagation in a model close to \eqref{eq:main}, where the {trait affects the growth rate $r$ but not the dispersal ability.} This latter assumption is made to take into account that the most favorable phenotypical trait may depend on space. The model reads
\begin{equation*}
\partial_t n - \Delta_{x,\theta} n = \left( r \left( \theta - B x \cdot e \right) - \int_{\R} k \left( \theta - B x \cdot e , \theta' - B x \cdot e \right) n(t,x,\theta') d \theta' \right) n(t,x,\theta),
\end{equation*}
and the authors prove the existence of travelling wave solutions. A version {with local competition in trait} of this equation has also been studied in \cite{Berestycki-Chapuisat}. {As compared to \cite{Alfaro,Berestycki-Chapuisat}, the main difficulty here is to obtain a uniform $L^\infty\left( \R \times \Theta \right)$ bound on the density $n$ solution of \eqref{eq:main}}. It is worth recalling that this propagation phenomena in reaction diffusion equations, through the theory of travelling waves, has been widely studied since the pioneering work of Aronson and Weinberger \cite{Aronson} on the Fisher-KPP equation \cite{Fisher,Kolmogorov}. We refer to \cite{Nadin,Nolen,Berestycki-Hamel} and the references therein for recent works concerning travelling waves for generalized Fisher-KPP equations in various heterogeneous media, and to \cite{Coville-Davila,Coville-Dupaigne,Shen} for works studying front propagation in models where the non locality appears in the dispersion operator.

Studying propagation phenomena in nonlocal equations can be pretty involved since some qualitative features like Turing instability may occur at the back of the front, see \cite{Berestycki-Nadin,Hamel}, due to lack of comparison principles. Nevertheless, it is sometimes still possible to {construct} travelling fronts with rather abstract arguments. In this article, we aim to give a complete proof of some formal results that were previously announced in \cite{Bouin}. Namely construct some travelling waves solutions of \eqref{eq:main} with the expected qualitative features at the edge of the front. Let us now give the definition of spatial travelling waves we seek for \eqref{eq:main}. 

\begin{definition}\label{defonde}

We say that a function $n(t,x,\theta)$ is a \textit{travelling wave} solution of speed $c \in \R^+$ in direction $e \in \mathbb{S}^n$ if it writes 
\begin{equation*}
\forall (t,x,\theta) \in \R^+ \times \R^n \times \Theta, \qquad n(t,x,\theta):= \mu \left( \xi:= x \cdot e - c t , \theta \right),
\end{equation*}
where {\textit{the profile}} $\mu \in \mathcal{C}_b^2 \left( \R \times \Theta \right)$ is nonnegative, satisfies 
\begin{equation*}
\liminf_{\xi \to - \infty} \mu \left( \xi , \cdot \right) > 0, \qquad \lim_{\xi \to + \infty} \mu \left( \xi , \cdot \right) = 0, 
\end{equation*}
and solves
\begin{equation}\label{eqkinwave}
\begin{cases}
- c \partial_{\xi} \mu = \theta {\partial_{\xi \xi} \mu} + \alpha \partial_{\theta  \theta} \mu + r  \mu  (1 - \nu), \qquad (\xi , \theta) \in \R \times \Theta, \medskip\\
\partial_\theta \mu(\xi,\theta_{\text{min}}) = \partial_\theta \mu(\xi,\theta_{\text{max}}) = 0, \qquad \xi \in \R.
\end{cases}
\end{equation} 
where $\nu$ is the macroscopic density associated to $\mu$, that is $\nu \left( \xi \right) = \int_\Theta \mu \left( \xi, \theta \right) d \theta$.

\end{definition}

To state the main existence result we first need to explain which heuristic considerations yield to the derivation of possible speeds for fronts. As for the standard Fisher-KPP equations, we expect that the fronts we build in this work are so-called \textit{pulled fronts}: They are driven by the dynamics of small populations at the edge of the front. In this case, the speed of the front can be obtained through the linearized equation of \eqref{eqkinwave} around {\bf $\mu <<1$}. The resulting equation (which is now {a local elliptic equation}) writes
\begin{equation}\label{eq:linmain}
\begin{cases}
- c \partial_{\xi} \widetilde \mu = \theta \partial_{\xi \xi} \widetilde\mu + \alpha \partial_{\theta\theta} \widetilde\mu + r \widetilde\mu, \qquad (\xi , \theta) \in \R \times \Theta, \medskip\\
\partial_\theta \widetilde\mu(\xi,\theta_{\text{min}}) = \partial_\theta \widetilde\mu(\xi,\theta_{\text{max}}) = 0, \qquad \xi \in \R.
\end{cases}
\end{equation} 

{Particular solutions of \eqref{eq:linmain} are a combination of an exponential decay in space and a monotonic profile in trait:} 
\begin{equation*}
\forall (\xi , \theta) \in \R \times \Theta, \qquad \widetilde\mu(\xi,\theta) = Q_\lambda(\theta) e^{- \lambda \xi},
\end{equation*}
where $\lambda > 0$ represents the spatial decreasing rate and $Q_\lambda$ the trait profile. {The pair} $(c(\lambda), Q_\lambda)$ solves the following \textit{spectral problem}:
\begin{equation}\label{eq:eigenpb}
\begin{cases} 
\alpha \partial_{\theta \theta} Q_\lambda(\theta) + \left( - \lambda c(\lambda) + \theta \lambda^2 + r \right) Q_\lambda(\theta) = 0\,, \qquad \theta\in \Theta, \medskip\\
\partial_\theta Q_\lambda \left( \theta_{\text{min}} \right) = \partial_\theta Q_\lambda \left( \theta_{\text{max}} \right) = 0, \medskip\\
Q_\lambda(\theta) > 0, \; \int_\Theta Q_\lambda(\theta)\, d\theta = 1\,. 
\end{cases}
\end{equation}

We refer to Section \ref{tools}, Proposition \ref{propspec} for detailed arguments showing that \eqref{eq:eigenpb} has a unique solution $\left( c(\lambda), Q_\lambda \right)$ for all $\lambda > 0$. We also prove there that we can define the minimal speed $c^*$ and its associated decreasing rate through the following formula:
\begin{equation}\label{minspeed}
c^*:= c(\lambda^*) = \min_{\lambda > 0} c(\lambda).
\end{equation}

\begin{remark}
We emphasize that this structure of spectral problem giving information about propagation in models of "kinetic" type is quite robust. We refer to \cite{Alfaro,Berestycki-Chapuisat,Bouin-Calvez-Nadin, Bouin-Calvez-Nadin-2} for works where this kind of dispersion relations also give the speed of propagation of possible travelling wave solutions, and to \cite{Bouin-2,Bouin-Calvez,Bouin-Mirrahimi} for recent works where the same kind of spectral problem appears to find the limiting Hamiltonian in the WKB expansion of hyperbolic limits.   
\end{remark}

We are now ready to state the main Theorem of this paper:

\begin{theorem}\label{wave}
Let {$\Theta:= \left( \theta_{\text{min}} , \theta_{\text{max}} \right), \theta_{\text{min}} > 0, \theta_{\text{min}} < + \infty$}  and $c^*$ be the minimal speed defined after \eqref{minspeed}. Then, there exists a travelling wave solution of \eqref{eq:main} of speed $c^*$ in the sense of Definition \ref{defonde} . 
\end{theorem}

This Theorem, together with the heuristic argument, has been announced in \cite{Bouin}.

\begin{remark}
As in \cite{Alfaro,Aronson}, we expect that waves going with faster speeds $c > c^*$ do exist and are constructible by a technique of sub- and super solutions. Nevertheless, since it {does not} make much difference with \cite{Alfaro}, we do not address this issue here.
\end{remark}

The paper is organized as follows. In Section \ref{tools}, {we study the spectral problem \eqref{eq:eigenpb} and provide some qualitative properties}. In Section \ref{Slab}, we elaborate a topological degree argument to solve the problem in a bounded slab. Finally in Section \ref{profileminspeed}, we construct the profile going with speed $c^*$ which proves the existence of Theorem \ref{wave}.  

{\section{The spectral problem.}\label{tools}}

We discuss the spectral problem naturally associated to \eqref{eq:main} that we have stated in \eqref{eq:eigenpb}. We state and prove some useful properties of $Q_{\lambda}$ and some relations between $c^*$ and $\lambda^*$. 

\begin{proposition}[{\bf {Qualitative} properties of the spectral problem}]\label{propspec}
For all $\lambda > 0$, the spectral problem \eqref{eq:eigenpb} has a unique solution $(c(\lambda),Q_\lambda)$. Moreover, the function $\lambda \mapsto c(\lambda)$ has a minimum, that we denote by $c^*$ and that we call the \textit{minimal speed}. This minimum is attained, and we denote by $\lambda^* > 0$ an associated decreasing rate and $Q_{\lambda^*}:= Q^*$ the corresponding profile. Then we have the following properties:

\begin{enumerate}

\item[(i)] {For all $\lambda > 0$,} the profile $Q_\lambda$ is increasing {w.r.t $\theta$}. {There exists $\theta_0$ such that} $Q_\lambda$ is convex on $\left[ \theta_{min} , \theta_{0} \right]$ and concave on $\left[ \theta_0 , \theta_{max} \right]$. Moreover, $\theta_0$ satisfies $- \lambda c(\lambda) + \lambda^2 \theta_0 + r = 0$

\item[(ii)] We define $\left\langle \theta_\lambda \right\rangle:=  \int_\Theta \theta Q_\lambda (\theta) d \theta$, the mean trait associated to the {decay} rate $\lambda$. {We also define $\left\langle \theta^* \right\rangle:=  \left\langle \theta_{\lambda^*} \right\rangle$}. One has
\begin{equation}\label{rel1}
\forall \lambda > 0, \qquad - \lambda c(\lambda) + \lambda^2 \left\langle \theta_\lambda \right\rangle + r = 0, \qquad \left\langle \theta_\lambda \right\rangle > \frac{\theta_{max} + \theta_{min}}{2}.
\end{equation}
\item[(iii)] {About the special features of the minimal speed, we have
\begin{equation}\label{rel6}
c^* > 2 \sqrt{r \langle \theta^* \rangle},
\end{equation}
\begin{equation}\label{rel4}
c^* \geq \lambda^* \left(  \theta_{max} + \theta_{min} \right).
\end{equation}}
\end{enumerate}
\end{proposition}

\begin{proof}[\bf Proof of Proposition \ref{propspec}]
We first prove the existence and uniqueness of $(c(\lambda),Q_{\lambda})$ for all positive $\lambda$. {Let $\beta > 0$ and $K$ be the positive cone of nonnegative functions in $\mathcal{C}^{1,\beta}\left( \Theta \right)$}. We define $L$ on $\mathcal{C}^{1,\beta}\left( \Theta \right)$ as below 
\begin{equation*}
L (u)= - \alpha \partial_{\theta\theta}u(\theta) - \left( \theta - \theta_{\text{max}} \right)\lambda^2 u (\theta).
\end{equation*} 

The resolvent of $L$ together with the Neumann boundary condition is compact from the regularizing effect of the Laplace term. Moreover, the strong maximum principle and the boundedness of $\Theta$ gives that it is strongly positive. Using the Krein-Rutman theorem we obtain that there exists a nonnegative eigenvalue {$\frac{1}{\gamma(\lambda)} $}, corresponding to a positive eigenfunction $Q_\lambda$. This eigenvalue is simple and none of the other eigenvalues corresponds to a positive eigenfunction. {As a consequence, $\lambda c(\lambda):= r+  \lambda^2 \theta_{\text{max}} - \gamma( \lambda )$ solves the problem}.

We come to the proof of $(i)$. {Since $Q_\lambda \in \mathcal{C}^2(\Theta)$ and satisfies Neumann boundary conditions}, there exists $\theta_0$ such that $\partial_{\theta \theta} Q_{\lambda}(\theta_0) = 0$. Since $- \lambda c(\lambda) + \lambda^2 \theta + r $ is increasing with $\theta$, the sign of $\partial_{\theta \theta} Q_{\lambda}$ and thus the {monotonicity} of $Q_\lambda$ follows. We deduce:
\begin{equation*}
\lambda^2 \theta_{min} + r  \leq \lambda c(\lambda)  \leq \lambda^2 \theta_{max} + r.
\end{equation*}
This yields 
\begin{equation*}
c(\lambda) \underset{\lambda \to 0}{\sim} \frac{r}{\lambda}, \qquad \lambda c(\lambda) = \mathcal{O}_{\lambda \to + \infty}(\lambda^2).
\end{equation*}
These latter relations and the continuity of $\lambda \mapsto c(\lambda)$ give the existence of a positive minimal speed $c^*$ and a smallest positive minimizer $\lambda^*$. 

We now prove $(ii)$. We obtain the first relation of \eqref{rel1} after integrating \eqref{eq:eigenpb} over $\Theta$ and recalling the Neumann boundary conditions. To get the second one, we divide the spectral problem by $Q_{\lambda}$ and then integrate over $\Theta$:
\begin{equation}\label{rel2}
\left\langle \theta_\lambda \right\rangle = \frac{\theta_{max} + \theta_{min}}{2} + \frac{\alpha}{\lambda^2 \vert \Theta \vert } \int_{\Theta} \left\vert \frac{\partial_\theta Q_{\lambda}}{Q_{\lambda}} \right\vert^2 d \theta > \frac{\theta_{max} + \theta_{min}}{2}.
\end{equation}

We finish with $(iii)$. {For this purpose, we define $W_{\lambda} = \left( Q_{\lambda} \right)^2 $. It satisfies Neumann boundary conditions on $\partial \Theta$ and 
\begin{equation*}
\forall \theta \in \Theta, \qquad \alpha \partial_{\theta \theta} W + 2 \left( - \lambda c(\lambda) + \lambda^2 \theta + r \right) W = \alpha \left( \frac{\partial_\theta W}{2 \sqrt{W}} \right)^2  \geq 0.
\end{equation*}
We thus deduce that 
\begin{equation*}
\lambda^2 \int_\Theta \theta W d \theta + \left( - \lambda c(\lambda) + r \right) \int_\Theta W d \theta > 0,
\end{equation*}
from which we deduce 
\begin{equation}\label{rel5}
\frac{\int_{\Theta} \theta \left( Q^* \right)^2 d \theta}{\int_{\Theta} \left( Q^* \right)^2 d \theta} > \left\langle \theta^* \right\rangle.
\end{equation} }

Differentiating \eqref{eq:eigenpb} with respect to $\lambda$, we obtain
\begin{equation*}
\left( - \lambda c'(\lambda) - c(\lambda) + 2 \theta \lambda \right) Q_\lambda + \left( - \lambda c(\lambda) + \theta \lambda^2 + r \right) \frac{\partial Q_\lambda}{\partial \lambda} + \alpha \partial_{\theta \theta} \left( \frac{\partial Q_\lambda}{\partial \lambda} \right) = 0.
\end{equation*}
We {do not} have {any} information about $\frac{\partial Q_\lambda}{\partial \lambda}$. Nevertheless, one can overcome this {issue by testing} directly {against} $Q_\lambda$. We obtain, for $\lambda = \lambda^*$:
\begin{equation*}
- c^* \int_\Theta \left( Q^* \right)^2 d\theta + 2  \lambda^* \int_\Theta \theta \left( Q^* \right)^2 d\theta = 0,
\end{equation*}
since $c'(\lambda^*)=0$. {As a consequence
\begin{equation}\label{rel3}
\qquad c^* = 2 \lambda^* \frac{\int_{\Theta} \theta \left( Q^* \right)^2 d \theta}{\int_{\Theta} \left( Q^* \right)^2 d \theta}.
\end{equation}
Combining \eqref{rel3} with $- \lambda^* c^* + \left(\lambda^*\right)^2 \left\langle \theta^* \right\rangle + r = 0$, one obtains
\begin{equation}
 \frac{(c^*)^2}{4r} = \frac12 \left( \frac{\int_{\Theta} \theta \left( Q^* \right)^2 d \theta}{\int_{\Theta} \left( Q^* \right)^2 d \theta}\right)^2  \left(  \frac{\int_{\Theta} \theta \left( Q^* \right)^2 d \theta}{\int_{\Theta} \left( Q^* \right)^2 d \theta} - \frac{\left\langle \theta^* \right\rangle}{2} \right)^{-1}.
\end{equation}
%
which gives \eqref{rel6} since $\frac12 \left( \frac{\int_{\Theta} \theta \left( Q^* \right)^2 d \theta}{\int_{\Theta} \left( Q^* \right)^2 d \theta}\right)^2  \left(  \frac{\int_{\Theta} \theta \left( Q^* \right)^2 d \theta}{\int_{\Theta} \left( Q^* \right)^2 d \theta} - \frac{\left\langle \theta^* \right\rangle}{2} \right)^{-1} \geq \left\langle \theta^* \right\rangle$ always holds true and \eqref{rel5} rules out equality.}

Finally, using \eqref{rel1} and \eqref{rel3}, one has  
\begin{equation*}
c^* > 2 \lambda^* \left\langle \theta^* \right\rangle \geq 2 \lambda^* \frac{\theta_{max} + \theta_{min}}{2} = \lambda^*\left( \theta_{max} + \theta_{min} \right).
\end{equation*}

\end{proof}

\section{Solving the problem in a bounded slab.}\label{Slab}

In this Section, we solve an approximated problem in a bounded slab $(-a , a) \times \Theta$.
\begin{definition}
For all $\tau > 0$, we define
\begin{equation*}
\forall \theta \in \Theta, \qquad g_\tau(\theta) = \theta_{min} + \tau \left( \theta - \theta_{min} \right).
\end{equation*}
Now, for all $a > 0$, the slab problem $P_{\tau,a}$ is defined as follows on $[-a , a] \times \Theta$:
\begin{equation}\label{eq:slab}
[P_{\tau,a}]\left\{\begin{array}{l}
-c \mu_{\xi}^a - g_{\tau}(\theta) \mu_{\xi \xi}^a - \alpha \mu_{\theta\theta}^a = r \mu^a (1 - \nu^a)\, , \quad {(\xi,\theta)} \in (-a,a) \times \Theta, \medskip\\
\mu_\theta^a(\xi,\theta_{\min}) = \mu_\theta^a(\xi,\theta_{\max}) = 0\, ,\quad \xi \in {(-a , a)}, \medskip\\
\mu^a(-a,\theta) = \vert\Theta \vert^{-1}\, , \quad \mu^a(a,\theta)  = 0\, , \quad \theta \in \Theta .
\end{array}
\right.
\end{equation}
with the supplementary renormalization condition $\nu^a(0) = \epsilon$. For legibility, we set $P_{1,a}:= P_a$.
\end{definition}

The non-local character of the source term does not provide any full comparison principle for $P_{\tau,a}$. However, we still have $\mu \geq0$. {We follow \cite{Alfaro,Berestycki-Nadin} and shall use the Leray-Schauder theory. For this purpose, some uniform \textit{a priori} estimates (with respect to $\tau, a$) on the solutions of the slab problem are required. The main difference with \cite{Alfaro,Berestycki-Nadin} is that it is more delicate to obtain these uniform $L^\infty$ estimates since it is not possible to write neither a useful equation nor an inequation on $\nu$ due to the term $\theta \mu_{\xi\xi}$ (as it is the case in kinetic equations)}. Our strategy is the following. We first prove in Lemma \ref{upboundc} that the speed is uniformly bounded from above. Then, Lemmas \ref{lem:nc=0} and \ref{bottom} focus on the case $c=0$ and prove that there cannot exist any solution to the slab problem in this case, provided that the normalization $\eps$ is well chosen. Finally, when the speed is given and uniformly bounded, we can derive a uniform \textit{a priori} estimate on the solutions of the slab problem \eqref{eq:slab}. Thanks to these \textit{a priori} estimates, we apply a Leray-Schauder topological degree argument in Proposition \ref{slabsol}. All along Section \ref{Slab}, we omit the superscript $a$ in $\mu^a$ {and $\nu^a$}.

\subsection{An upper bound for $c$.}

\begin{lemma}\label{upboundc}
For any normalization parameter $\epsilon > 0$, there exists a sufficiently large $a_0(\epsilon)$ such that any {pair} $(c,\mu)$ solution of the slab problem $P_{\tau,a}$ with $a \geq a_0(\eps)$ satisfies $c \leq c_\tau^* \leq c^*$.
\end{lemma}

\begin{proof}[{\bf Proof of Lemma \ref{upboundc}}]
We just adapt an argument from \cite{Alfaro,Berestycki-Nadin}. It consists in finding a relevant subsolution for a related problem. 
As $\mu \geq 0$, one has
\begin{equation}\label{eq:n}
\forall (\xi,\theta) \in (-a,a) \times \Theta, \qquad -c \mu_{\xi} \leq g_\tau(\theta) \mu_{\xi\xi} + \alpha \mu_{\theta\theta} + r \mu.
\end{equation}
As \eqref{eq:eigenpb}, the following pertubated spectral problem has a unique solution associated with a minimal speed $c_\tau^*$: 
\begin{equation}
\begin{cases} 
\alpha \partial_{\theta\theta} Q_\tau^*(\theta) + \left( - \lambda_\tau^* c_\tau^* + g_\tau(\theta) \left( \lambda_\tau^* \right)^2 + r \right) Q_\tau^*(\theta) = 0\, , \qquad \theta\in \Theta, \medskip\\
\partial_\theta Q_\tau^* \left( \theta_{\text{min}} \right) = \partial_\theta Q_\tau^* \left( \theta_{\text{max}} \right) = 0, \medskip\\
Q_\tau^*(\theta) > 0, \; \int_\Theta Q_\tau^*(\theta)\, d\theta = 1\,. 
\end{cases}
\end{equation}
Let us assume by contradiction that $ c > c^* $, then the family of functions $\psi_A ( \xi, \theta ):= A {e^{- \lambda_\tau^*\xi} Q_\tau^* ( \theta )}$ verifies
\begin{equation}\label{eq:psi}
\forall (\xi,\theta) \in (-a,a) \times \Theta, \qquad  g_\tau(\theta) \left( \psi_A \right)_{\xi\xi} + \alpha \left( \psi_A \right)_{\theta\theta} + r \psi_A =  \lambda^* c^* \psi_A < - c \left( \psi_A \right)_\xi,
\end{equation}
As the eigenvector $Q^*$ is positive, and $\mu \in L^{\infty} \left( -a , a \right)$, one has $\mu \leq \psi_A$ for $A$ sufficiently large, and $\mu \geq \psi_A$ for $A$ sufficiently small. As a consequence, one can define 
\begin{equation*}
A_0 = \inf \left\lbrace A \; \vert \; \forall (\xi, \theta) \in \left( -a , a \right) \times \Theta, \; \psi_A (\xi, \theta) > \mu(\xi,\theta) \right\rbrace. 
\end{equation*}
Necessarily, $A_0 > 0$ and there exists a point $(\xi_0, \theta_0) \in \left[ -a , a \right] \times \left[ \theta_{\text{min}}, \theta_{\text{max}} \right] $ where $\psi_{A_0}$ touches $\mu$:
\begin{equation*}
\mu(\xi_0 , \theta_0) = \psi_{A_0}(\xi_0 , \theta_0).
\end{equation*}
This point minimizes $\psi_A - n $ and {cannot} be in $\left( -a , a \right) \times \Theta$. Indeed, combining  \eqref{eq:n} and \eqref{eq:psi}, one has in the interior, 
\begin{equation*}
{\forall (\xi,\theta) \in ( -a , a ) \times \Theta} , \qquad  c \left( \psi_A - n \right)_\xi + g_\tau(\theta) \left( \psi_A - n \right)_{\xi\xi} + \alpha \left( \psi_A - n \right)_{\theta\theta} + r \left( \psi_A - n \right)  < 0.
\end{equation*}
But, if $(\xi_0,\theta_0)$ is in the interior, this latter inequality cannot hold since $ \theta \left( \psi_A - n \right)_{\xi\xi} + \alpha \left( \psi_A - n \right)_{\theta\theta} \geq 0 $.
{Next we} eliminate the boundaries. First, $(\xi_0,\theta_0)$ cannot lie in the right boundary $ \left\lbrace x = a \right \rbrace \times \Theta $ since $\psi_{A_0} > 0$ and $ \mu = 0 $ there. Moreover, thanks to the Neumann boundary conditions satisfied by both $\psi_A$ and $\mu$, $(\xi_0,\theta_0)$ {cannot} be in $\left[ -a , a \right] \times \left\lbrace \theta_{\text{min}}, \theta_{\text{max}} \right \rbrace$. We now exclude the left boundary by adjusting the normalization. If $\xi_0 = -a$, then $\psi_A (\xi_0, \theta_0) = \vert \Theta \vert^{-1}$ and {$A_0 = \frac{e^{- \lambda_\tau^* a}}{\vert \Theta \vert Q_\tau^* ( \theta_0 ) }$. Then $\nu(0) \leq \frac{e^{- \lambda_\tau^* a}}{\Theta Q_\tau^* ( \theta_0 ) }$ which is smaller than $\epsilon$ for a sufficiently large $a$.}

\end{proof}
\subsection{The special case $c=0$.}

We now focus on the special case $c=0$. We first show (Lemma \ref{lem:nc=0}) that the density $\mu$ is uniformly bounded (with respect to $a >0$). From this estimate, we deduce in Lemma \ref{bottom} that there exists a constant $\eps_0$ depending only on the fixed parameters of the problem such that necessarily $\nu(0) \geq \eps_0$. Thus, provided that $\eps$ is set sufficiently small, our analysis will conclude that the slab problem {does not admit a} solution of the form $(c,\mu) = (0,\mu)$ {for $\eps < \eps_0$. We emphasize that the key \textit{a priori} estimate, \textit{i.e.} $\nu \in L^\infty\left( (-a,a) \times \Theta \right)$, is easier to obtain in the case $c=0$ than in the case $c \neq 0$ (compare Lemmas \ref{lem:nc=0} and \ref{lem:nc}}).

\subsubsection{A priori estimate for {$\mu$} when $c=0$.}

\begin{lemma}{\bf (A priori estimates, $c = 0$).}\label{lem:nc=0}

Assume $c = 0$, $b > 0$ and $\tau \in [0,1]$. {There exists a constant $C(b)$ such that every solution $(c=0,\mu)$ of \eqref{eq:slab} satisfies}
\begin{equation*}
\forall (\xi,\theta) \in [ -b , b ] \times \Theta, \quad \mu(\xi,\theta) \leq \frac{{C(b)}}{\Theta} \frac{\theta_{max}}{\theta_{min}}.
\end{equation*}
\end{lemma}

\begin{proof}[{\bf Proof of Lemma \ref{lem:nc=0}}]
When $c=0$, the slab problem \eqref{eq:slab} reduces to 
\begin{equation*}
[P_{\tau,b}]\left\{\begin{array}{l}
- g_{\tau}(\theta) \mu_{\xi \xi} - \alpha \mu_{\theta\theta} = r \mu (1 - \nu)\, , \quad (\xi,\theta) \in {(-b,b)} \times \Theta, \medskip\\
\mu_\theta(\xi,\theta_{\min}) = \mu_\theta(\xi,\theta_{\max}) = 0\, ,\quad \xi \in {(-b , b)}, \medskip\\
\mu(-b,\theta) = \vert \Theta \vert^{-1}\, , \quad \mu(b,\theta)  = 0\, , \quad \theta \in \Theta .
\end{array}
\right.
\end{equation*}
Integration with respect to the trait variable $\theta$ yields
\begin{equation*}
\left\{\begin{array}{l}
- \displaystyle \left( \int_{\Theta} g_\tau(\theta) \mu(x,\theta) d\theta \right)_{\xi\xi} = r \nu(\xi) (1 - \nu(\xi)), \quad \xi \in \R,\\
\nu(-b) = 1\, , \quad \nu(b)  = 0\, .
\end{array}
\right.
\end{equation*}
Take a point $\xi_0$ where $\int_{\Theta} g_\tau(\theta) \mu(\xi,\theta) d\theta$ attains a maximum. At this point, one has necessarily $\nu(\xi_0) \leq 1$. The following sequence of inequalities {holds true} for all $\xi \in {\left(-b,b\right)}$:
\begin{multline*}
{\theta_{\text{min}} \nu(\xi)} = g_\tau(\theta_{min}) \nu(\xi) = g_\tau(\theta_{min})  \int_\Theta \mu(\xi,\theta) d\theta \leq \int_{\Theta} g_\tau(\theta)  \mu(\xi,\theta) d\theta \\ \leq \int_{\Theta} g_\tau(\theta) \mu(x_0,\theta) d\theta \leq g_\tau(\theta_{max}) \nu(x_0) \leq g_\tau(\theta_{max}),
\end{multline*}
and give
\begin{equation*}
{\forall \xi \in ( -b , b ), \quad \nu(\xi) \leq \frac{g_\tau(\theta_{max})}{\theta_{min}} \leq \frac{\theta_{max}}{\theta_{min}}.}
\end{equation*}
Now, the Harnack inequality of Proposition \ref{Harnack} gives
\begin{equation*}
\forall (\xi,\theta) \in (-b,b) \times \Theta, \qquad n(\xi,\theta) \leq \frac{C(b)}{\vert \Theta \vert } \nu(\xi)  \leq \frac{C(b)}{\vert \Theta \vert} \frac{\theta_{max}}{\theta_{min}}.
\end{equation*}
\end{proof}

\subsubsection{Non-existence of solutions of the slab problem when $c=0$.}

\begin{lemma}{\bf (Lower bound for $\nu(0)$ when $c=0$).}\label{bottom}
There exists $\epsilon_0 > 0$ such that if $a$ is large enough, then for all $\tau \in [0,1]$, any solution of the slab problem $(c=0,\mu)$ satisfies $\nu(0) > \epsilon_0$. 
\end{lemma}

\begin{proof}[{\bf Proof of Lemma \ref{bottom}}]
We adapt {an} argument from {\cite{Alfaro}}. It is a bit simpler here since {the} trait space is bounded. For $b>0$, consider the following spectral problem in both variables $(\xi,\theta)$:
{\begin{equation}\label{eq:evpb}
\left\{\begin{array}{l}
g_\tau(\theta) \left( \varphi_b \right)_{\xi\xi} + \alpha \left( \varphi_b \right)_{\theta\theta} + r \varphi_b   = \psi_b \varphi_b \,, \quad (\xi,\theta) \in \left( - b , b\right) \times \Theta,\medskip\\
\left( \varphi_b \right)_\theta(\xi,\theta_{\min}) = \left( \varphi_b \right)_\theta(\xi,\theta_{\max}) = 0\,, \quad \xi \in \left( - b , b\right),\medskip\\
\varphi_b(-b,\theta) = 0 \, , \quad \varphi_b(b,\theta)  = 0\,, \quad \theta \in \Theta.
\end{array}
\right.
\end{equation}
One can rescale the problem in the space direction setting $\xi = b \zeta$:
\begin{equation}\label{eq:evpb}
\left\{\begin{array}{l}
\dfrac{g_\tau(\theta)}{b^2} \left( \varphi_b \right)_{\zeta \zeta} + \alpha \left( \varphi_b \right)_{\theta\theta} + r \varphi_b   = \psi_b \varphi_b \,, \quad (\zeta,\theta) \in \left( - 1 , 1\right) \times \Theta \, ,\medskip\\
\left( \varphi_b \right)_\theta(\zeta,\theta_{\min}) = \left( \varphi_b \right)_\theta(\zeta,\theta_{\max}) = 0\,,\quad  \zeta \in \left( - 1 , 1\right) ,\medskip\\
\varphi_b(-1,\theta) = 0 \, , \quad \varphi_b(1,\theta)  = 0\,, \quad \theta \in \Theta.
\end{array}
\right.
\end{equation}}
{Using an Hamilton-Jacobi technique (see for instance \cite{Bouin-Mirrahimi} and all the references therein), one can prove that $\lim_{b \to +\infty} \psi_b = r$.} 
As a consequence, we fix $b$ sufficiently large to have $\psi_b > \frac{r}{2}$.

Thanks to the \textit{a priori} estimate on $\mu$ obtained in Lemma \ref{lem:nc=0}, {and} by the Harnack inequality {(of Proposition \ref{Harnack})}, there exists a constant $C(b)$ which does not depend on $a > b$ such that
\begin{equation*}
\forall \theta \in \Theta, \qquad C(b) \mu(0, \theta) \geq C(b) \inf_{\left( -b , b\right) \times \Theta} \mu(\xi , \theta) \geq \|\mu\|_{L^\infty((-b,b)\times \Theta)}.
\end{equation*}
To compare \eqref{eq:slab} to \eqref{eq:evpb}, one has, for all $(\xi,\theta) \in [-b,b] \times \Theta$,
\begin{equation*}
g_\tau(\theta) \mu_{\xi\xi} + \mu_{\theta\theta} + r \mu = r \mu \nu \leq  r \mu \vert \Theta \vert \|\mu\|_{L^\infty((-b,b)\times \Theta)} \leq  r C \nu(0) \mu(\xi,\theta). 
\end{equation*}
We deduce from this computation that as soon as $\nu(0) \leq \frac{1}{2C(b)}$, one has 

\begin{equation*}
\forall (\xi,\theta) \in [-b,b] \times \Theta, \quad r C \nu(0) \mu(\xi,\theta) < {\psi_b} \mu(\xi,\theta),
\end{equation*}
and this means that $\mu$ is a subsolution of \eqref{eq:evpb}. We can now use the same arguments as for the proof of Lemma \ref{upboundc}. We define
\begin{equation*}
A_0 = \max \left\lbrace A \; \vert \; \forall (\xi, \theta) \in \left[ -b , b \right] \times \Theta, \; A \varphi_b (\xi, \theta) < \mu(\xi,\theta) \right\rbrace,
\end{equation*}
so that $u_{A_0}:= \mu - A_0 \varphi_b$ has a zero minimum in $(\xi_0,\theta_0)$ and satisfies 
\begin{equation*}
\left\{\begin{array}{l}
- g_\tau(\theta) \left( u_{A_0} \right)_{\xi\xi} - \alpha \left(u_{A_0} \right)_{\theta\theta} - r u_{A_0}  > - \psi_b u_{A_0}\,, \quad (\xi,\theta) \in \left( - b , b\right) \times \Theta \, ,\medskip\\
\left( u_{A_0} \right)_\theta(\xi,\theta_{\min}) = \left( u_{A_0} \right)_\theta(\xi,\theta_{\max}) = 0\,, \quad \xi \in \left( - b , b\right) ,\medskip\\
u_{A_0}(-b,\theta) > 0 \, , \quad u_{A_0}(b,\theta)  > 0\,, \quad \theta \in \Theta.
\end{array}
\right.
\end{equation*}
For the same reasons as in Lemma \ref{upboundc} this cannot hold, so that necessarily {$ \nu(0) > \eps_0:= \frac{1}{2C(b)} $}.

\end{proof}

\subsection{Uniform bound over the steady states, for $c \in \left[ 0 , c^* \right]$.}

The previous Subsection is central in our analysis. Indeed, it gives a bounded set of speeds where to apply the Leray-Schauder topological degree argument, namely we can restrict ourselves to speeds $c \in \left[ 0 , c^* \right]$. Based on this observation, we are now able to derive a uniform $L^\infty$ estimate (with respect to $a$ and $\tau$) for solutions $\mu$ of \eqref{eq:slab}. This is done in Lemma \ref{lem:nc} below.

\begin{lemma}{\bf (A priori estimates, $c \in \left[ 0 , c^* \right]$).}\label{lem:nc}

Assume $c \in \left[ 0 , c^* \right]$, $\tau \in [0,1]$ and $a\geq 1$. Then there exists a constant $C_0$ depending only on $\theta_{\min}$ and  $\vert \Theta \vert$ such that any solution $(c,\mu)$ of the slab problem $P_{a,\tau}$ satisfies
\begin{equation*}
\Vert \mu \Vert_{L^\infty\left( (-a,a) \times \Theta \right) }\leq C_0\,.
\end{equation*}
\end{lemma}

\begin{proof}[{\bf Proof of Lemma \ref{lem:nc}}]

We divide the proof into two steps. In the first step, we prove successively that $\mu$ and $\mu_\theta$ are bounded uniformly in $H^1\left( (-a,a) \times \Theta \right)$. In the second step, we use a suitable trace inequality to deduce a uniform $L^\infty\left( (-a,a) \times \Theta \right)$ estimate on $\mu$. We define $K_0(a) = \max_{[-a,a] \times \Theta} \mu$. We want to prove that $K_0(a)$ is in fact bounded uniformly in $a$.

{The argument is inspired from \cite{Berestycki-Nadin}. The principle of the proof goes as follows: The maximum principle implies that $\nu(\xi_0) \leq 1$ if $\left( \xi_0 , \theta_0 \right)$ is a maximum point for $\mu$. This does not imply that $\max \mu \leq 1$. However, we can control $\mu(\xi_0,\theta_0)$ by the non local term $\nu(\xi_0)$ providing some regularity of $\mu$ in the direction $\theta$. In order to get this additional regularity we use the particular structure of the equation (the nonlocal term does not depend on $\theta$ and is non negative).}

\quad

{\bf \# Step 0: Preliminary observations.}

\quad

Denote by $(\xi_0,\theta_0)$ a point where the maximum is reached. If the maximum is attained on the $\xi-$boundary $\xi_0 = \pm a$ then $K_0(a) \leq \vert \Theta \vert^{-1}$ by definition. If it is attained on the $\theta-$boundary $\theta_0 \in \{\theta_{\min},\theta_{\max}\}$, then the first derivative $\partial_\theta \mu$ vanishes by definition. Hence $\mu_{\theta\theta}(\xi_0,\theta_0)\leq 0$ and $\mu_{\xi\xi}(\xi_0,\theta_0)\leq 0$. The same holds true if $(\xi_0,\theta_0)$ is an interior point. Evaluating equation \eqref{eq:slab} at $(\xi_0,\theta_0)$ implies
\begin{equation*}
K_0(a) (1 - \nu(\xi_0))\geq 0\,,
\end{equation*}
and therefore $\nu_0(\xi_0) \leq 1$.

\quad

{\bf \# Step 1: Energy estimates on $\mu$.}

\quad

We derive local energy estimates.
We introduce a smooth cut-off function $\chi: \R \to [0,1]$ such that
\begin{equation*}
\begin{cases}
\chi = 1 \qquad \text{on} \qquad J_1 = \left( \xi_0 - \frac12,\xi_0 + \frac12 \right),  \medskip\\
\chi = 0 \qquad \text{outside} \qquad  J_2 = \left[\xi_0 - 1,\xi_0 + 1\right]. 
\end{cases}
\end{equation*} 
 Notice that the support of the cut-off function does not necessarily avoid the $\xi-$boundary. We also introduce the following linear corrector 
\begin{equation*}
\forall \xi \in [-a,a], \qquad m(\xi) = \frac{1}{\vert \Theta \vert} \frac{a - \xi}{2a},
\end{equation*}
which is defined such that $m(-a) = \vert \Theta \vert^{-1}$, $m(a) = 0$, and $0\leq m\leq \vert \Theta \vert^{-1}$ on $(-a,a)$. 
Testing against $(\mu - m)\chi$ over $[-a,a] \times \Theta$, we get
\begin{multline*}
- c \int_{(-a,a) \times \Theta} (\mu - m) \chi \mu_\xi d\xi d\theta -\int_{(-a,a) \times \Theta} g_\tau(\theta) (\mu - m)_{\xi\xi} (\mu - m)\chi\, d\xi d\theta \\- \int_{(-a,a) \times \Theta} \mu_{\theta\theta} (\mu-m)\chi\, d\xi d\theta =\int_{(-a,a) \times \Theta} \mu(1 - \nu)( \mu-m)\chi\, d\xi d\theta.
\end{multline*}
We now transform each term of the l.h.s. by integration by parts. We emphasize that the linear correction $m$ ensures that all the boundary terms vanish. We get 
\begin{multline*}
\int_{(-a,a) \times \Theta} g_\tau(\theta)  \left|(\mu -m)_{\xi}\right|^2 \chi\, d\xi d\theta + \int_{(-a,a) \times \Theta} \left|\mu_\theta\right|^2 \chi  \, d\xi d\theta \\ \leq \frac12 \int_{(-a,a) \times \Theta} g_\tau(\theta) (\mu-m)^2 \chi_{\xi\xi}\, d\xi d\theta  + c \frac{\vert \Theta \vert^{-1}}{2a} \int_{(-a,a) \times \Theta}  \chi (\mu-m) d\xi d\theta \\- c \int_{(-a,a) \times \Theta} \frac12 (\mu-m)^2 \chi_\xi d\xi d\theta + \int_{(-a,a) \times \Theta}  \mu^2 \chi\, d\xi d\theta + \int_{(-a,a) \times \Theta}  \mu \nu m\chi\, d\xi d\theta.
\end{multline*}
We use that $\mu \leq K_0(a)$, $\nu(\xi) \leq \vert\Theta\vert K_0(a)$, $g_\tau(\theta) \geq \theta_{\text{min}}$ and $\vert c \vert \leq c^*$ to get
\begin{multline*}
\theta_{\min}\int_{J_1 \times \Theta}  \left|\mu_\xi-m_{\xi}\right|^2 \, d\xi d\theta + \int_{J_1 \times \Theta} \left|\mu_\theta\right|^2   \, d\xi d\theta \\ \leq c^* \frac{\vert \Theta \vert^{-1}}{2a} K_0 \vert J_2 \times \Theta \vert  - c \int_{[-a,a] \times \Theta} \frac12 (\mu-m)^2 \chi_\xi d\xi d\theta \\+ \frac12\int_{(-a,a) \times \Theta} g_\tau(\theta) (\mu-m)^2    \chi_{\xi\xi}\, d\xi d\theta + \int_{J_2 \times \Theta}  K_0^2 \, d\xi d\theta +   \int_{J_2 \times \Theta} \vert \Theta \vert K_0^2 \, d\xi d\theta\,,
\end{multline*}
Then we use the pointwise inequality $| \mu_\xi - m_\xi |^2 \geq \mu_\xi^2/2 - m_\xi^2$ in the first integral of the l.h.s.:
\begin{multline*}\label{testn2}
\frac{\theta_{\min}}2\int_{J_1}  \left|\mu_{\xi}\right|^2 \, d\xi d\theta + \int_{J_1} \left|\mu_\theta\right|^2 \, d\xi d\theta   \leq \frac{K_0 c^*}{a}  + \theta_{\min} \int_{J_1}  \left|m_{\xi}\right|^2 \, d\xi d\theta \\\ + \int g_{\tau}(\theta)\left( \mu^2 + m^2 \right)    \chi_{\xi\xi}\, d\xi d\theta + c^* \int  \left( \mu^2 + m^2 \right) \chi_\xi d\xi d\theta +    4 \vert \Theta \vert K_0^2. 
\end{multline*}
{Thus, we obtain our first energy estimate: $\mu \in H^1\left( [-a,a] \times \Theta \right)$ with a uniform bound of order $\mathcal{O}\left( K_0(a)^2 \right)$ uniformly:}
\begin{equation}\label{nH1}
\min\left( \dfrac{\theta_{\min}}2 , 1\right)\int_{J_1}  \left(\left|\mu_{\xi}\right|^2 +   \left|\mu_\theta\right|^2\right) \, d\xi d\theta  \leq    C(\vert \Theta \vert,\theta_{\min},\chi)\left( 1 + K_0(a)^2\right)\,,
\end{equation}
as soon as $a \geq \frac{1}{2}$. 

We now come to the proof that $\partial_\theta \mu$ is also in $H^1\left( (-a,a) \times \Theta \right)$. We differentiate \eqref{eq:slab} with respect to $\theta$ for this purpose. Here, we use crucially that $\nu$ is a function of the variable $x$ only. Note that we cannot expect that $\mu \in H^2\left( [-a,a] \times \Theta \right)$ with a bound of order $\mathcal{O}\left( K_0(a)^2 \right)$ at this stage. But we need additional elliptic regularity in the variable $\theta$ only.
\begin{equation} \label{eq:n_theta}
\forall (\xi,\theta) \in (-a,a) \times \Theta, \qquad - c \mu_{\xi\theta} - \tau \mu_{\xi\xi} - g_\tau(\theta) \mu_{\xi\xi\theta} - \mu_{\theta\theta\theta} = \mu_\theta (1 - \nu)\, .
\end{equation}
We use the cut-off function $\widetilde \chi(\xi) = \chi(\xi_0 + 2(\xi-\xi_0))$, for which $\supp \widetilde \chi\subset J_1$, and $\chi(\xi) = 1$ on $J_{1/2} = (\xi_0 - 1/4,\xi_0 + 1/4)$. Multiplying \eqref{eq:n_theta} by $\mu_\theta\widetilde \chi$, we get after integration by parts
{\begin{multline*}\label{testntheta2}
\int_{J_1} \tau \mu_\xi \mu_{\theta \xi} \widetilde \chi  \, d\xi d\theta + \int_{J_1} \tau \mu_\xi \mu_{\theta } \widetilde \chi_\xi  \, d\xi d\theta + \int_{J_1} g_\tau(\theta) \mu_{\xi \theta} \mu_\theta \widetilde \chi_\xi  \, d\xi d\theta \\ + \int_{J_1} g_\tau(\theta) \left|\mu_{\xi \theta}\right|^2 \widetilde \chi  \, d\xi d\theta  +  \int_{J_1}   \left|\mu_{\theta\theta}\right|^2 \widetilde \chi  \, d\xi d\theta \leq \int_{J_1}   \left|\mu_{ \theta}\right|^2 \widetilde \chi  \, d\xi d\theta\,+ c \int_{J_1} \tilde \chi_\xi \frac{\vert \mu_\theta \vert^2}{2} d\xi d\theta .
\end{multline*}}
Notice that all the boundary terms vanish since $\mu_\theta = 0$ on all segments of the boundary. Using the $H^1$ estimate \eqref{nH1} obtained previously for $\mu$, we deduce
\begin{multline*}
\frac{\theta_{\min}}{2} \int_{J_{1/2}} \left \vert \mu_{\theta \xi}\right\vert^2   \, d\xi d\theta  +  \int_{J_{1/2}}   \left \vert \mu_{\theta\theta}\right \vert^2   \, d\xi d\theta \leq \left( 1 + \frac{c^*}{2} \Vert \widetilde\chi_\xi \Vert_\infty \right)\int_{J_1}   \left \vert \mu_{ \theta}\right \vert^2    \, d\xi d\theta + \frac{1}{2\theta_{\min}} \int_{J_{1}}   \left\vert \mu_\xi \right\vert^2  d\xi d\theta \\ + \frac12 \int_{J_1}\left( \left|\mu_\xi \right|^2 + \left|\mu_\theta\right|^2 \right) \left|\widetilde \chi_\xi \right|  \, d\xi d\theta 
 + \frac12 \int \theta \left|\mu_{\theta}\right|^2  \widetilde \chi_{\xi\xi}  \, d\xi d\theta
\end{multline*}
from which we conclude
\begin{equation}\label{nthetaH1}
\min\left( \dfrac{\theta_{\min}}2 , 1\right)\int_{J_1}  \left(\left|\mu_{\theta \xi}\right|^2 +   \left|\mu_{\theta\theta}\right|^2\right) \, d\xi d\theta  \leq    \overline{C}(\Theta,\theta_{\min},\chi)\left( 1 + K_0(a)^2\right)\, .
\end{equation}
This crucial computation proves that $\mu_\theta$ also belongs to $H^1\left( (-a,a) \times \Theta \right)$.

\quad 

{\bf \# Step 2: Improved regularity of the trace $\mu(\xi, \cdot )$.}

\quad

From the fact that $\mu_\theta$ is a $H^1\left( (-a,a) \times \Theta \right)$ function uniformly in $a$, we obtain that the trace function $\theta \mapsto \mu_\theta(\xi_0,\theta)$ belongs to $H^{1/2}\left( \Theta \right)$ uniformly by standard trace theorems. Therefore, $\theta \mapsto \mu(\xi_0,\theta)$ belongs to $H^{3/2}\left( \Theta \right)$ . This gives a constant $C_{tr}$ such that 
\begin{equation*}
\Vert \mu(\xi_0,\cdot) \Vert_{H^{3/2}_\theta}^2 \leq C_{tr} \Vert \mu_{\theta} \Vert_{H^{1}_{x,\theta}}^2
\end{equation*}
This enables to control the variations of the density $\mu$ in the direction $\theta$. 

Indeed, by interpolation inequality there exists a constant $C_{\text{int}}$ such that in the variable $\theta$, at a given point $\xi_0$:
\[
\begin{cases}
\|\mu\left( \xi_0, \cdot \right)\|_{L^\infty_\theta}^3 \leq C_{int} \|\mu\left( \xi_0, \cdot \right)\|_{L^1_\theta} \|\mu\left( \xi_0, \cdot \right)\|_{H^{3/2}_\theta}^2 & \mathrm{if}\quad \dfrac{\|\mu\left( \xi_0, \cdot \right)\|_{L^1_\theta}}{\|\mu\left( \xi_0, \cdot \right)\|_{H^{3/2}_\theta}} \leq \dfrac{1}{C_{int}} \medskip ,\\
\|\mu\left( \xi_0, \cdot \right)\|_{L^\infty_\theta} \leq C_{int} \|\mu\left( \xi_0, \cdot \right)\|_{L ^1_\theta} & \mathrm{otherwise,} 
\end{cases} \]
(we refer to the Appendix for a proof of this inequality). Recall that $\nu(\xi_0) = \Vert \mu(\xi_0,\cdot) \Vert_{L^1_\theta} \leq 1$. It yields, combining with estimates \eqref{nH1} and  \eqref{nthetaH1} of {\bf \# Step 1}:
\[
\begin{cases}
K_0(a)^3 \leq C_{int} C_{tr} \overline{C} \nu(\xi_0) \left(1 + K_0(a)^2 \right)& \mathrm{if}\quad \dfrac{\nu(\xi_0)}{\| \mu(\xi_0,\cdot)\|_{H^{3/2}_\theta}} \leq \dfrac{1}{C_{int}}, \medskip \\
K_0(a) \leq C_{int} \nu(\xi_0) & \mathrm{otherwise}. 
\end{cases} \]
In both cases, this bounds $K_0(a)$ uniformly with respect to $a >0$.
This concludes the proof of Lemma \ref{lem:nc}.

\end{proof}
 
\subsection{Resolution of the problem in the slab.}

We now finish the proof of the existence of solutions of \eqref{eq:slab}. As previously explained, it consists in a Leray-Schauder topological degree argument. All uniform estimates derived in the previous Sections are key points to obtain \textit{a priori} estimates on steady states of suitable operators. We then simplify the problem with homotopy invariances. {We begin with a very classical problem: the construction of KPP travelling waves for the Fisher-KPP equation in a slab.}
\begin{lemma}\label{propKPP}
Let us consider the following Fisher-KPP problem in the slab $(-a,a)$:
\begin{equation*}
\left\{\begin{array}{l}
-c \nu_{\xi} -  \theta_{min} \nu_{\xi\xi} = r \nu (1 - \nu )\,, \qquad \xi \in \left( -a,a \right), \medskip\\
 \nu(-a) = 1\, , \quad \nu(a)  = 0\, .
\end{array}
\right.
\end{equation*}

One has the following properties:

\begin{enumerate}
\item For a given $c$, there exists a unique decreasing solution {$\nu^c \in [0,1]$. Moreover, the function $c \to \nu^c$ is continuous and decreasing.} 
\item\label{2} There exists $\eps^* > 0$ (independent of $a$) such that all solution with $c = 0$ satisfies $\nu_{c=0}(0) > \eps^*$. 
\item\label{3} For all $\eps > 0$, there exists $a(\eps)$ such that for all {$c > 2 \sqrt{r \theta_{min}}$, $\nu (0) < \eps$} for $a \geq a(\eps)$. 
\item {As a corollary of \ref{2} and \ref{3}, for all $\eps < \eps^*$, there exists a unique $c_0 \in \, [0, 2 \sqrt{r \theta_{min}}]$} such that $\nu_{c_0}(0) = \eps$ for $a \geq a(\eps)$.  
\end{enumerate}
\end{lemma}

\begin{proof}[{\bf Proof of Lemma \ref{propKPP}}]
The existence and uniqueness of solutions follows from \cite{Aronson}. {Again by maximum principle arguments, $\nu \in [0,1]$}. The solution is necessarily decreasing since
\begin{equation*}
\forall \xi \in \, (-a,a), \qquad \left( \nu_\xi e^{\frac{c}{\theta_{min}} \xi} \right)_\xi \leq 0, 
\end{equation*}
and $\nu_\xi(-a) \leq 0$.
By classical arguments, the application {$c \to \nu^c$} is continuous. For the decreasing character, we write, for $c_1 < c_2$ and $v:= \nu_2 - \nu_1$:
\begin{equation*}
- c_2 v_\xi   - \theta_{min} v_{\xi\xi} = {\left( 1 - \left( \nu_1 + \nu_2 \right) \right)  v }+ \left( c_2 - c_1 \right) \left( \nu_1 \right)_\xi ,
\end{equation*}
{so that $v$ satisfies
\begin{equation*}
\left\{\begin{array}{l}
- c_2 v_\xi   - \theta_{min} v_{\xi\xi} \leq  \left( 1 - \left( \nu_1 + \nu_2 \right) \right)  v , \qquad \xi \in \left( -a,a \right), \medskip\\
 v(-a) = 0\, , \quad v(a)  = 0\, .
\end{array}
\right.
\end{equation*} 
The maximum principle then yields that $v \leq 0$, that is $\nu_2 \leq \nu_1$}. The proofs of Lemmas \ref{upboundc} and \ref{bottom} can be adapted to prove the remainder of the Lemma. 

\end{proof}

With this $\eps^*$ in hand, we can state the main Proposition:

\begin{proposition}{\bf (Solution in the slab).}\label{slabsol}
Let $\epsilon < \min\left( \eps_0 , \eps^*\right)$. There exists $C_0 > 0$ and $a_0(\epsilon) > 0$ such that for all $a \geq a_0$, the slab problem \eqref{eq:slab}
with the normalization condition $\nu(0) = \epsilon$ has a solution $(c,\mu)$ such that
\begin{equation*}
\Vert \mu \Vert_{L^\infty \left(  [-a,a] \times \Theta \right)} \leq C_0, \qquad c \in \left] \, 0 , c^{*} \, \right]. 
\end{equation*}
\end{proposition}

\begin{proof}[{\bf Proof of Proposition \ref{slabsol}}]
Given a non negative function $\mu(\xi,\theta)$ satisfying the boundary conditions 
\begin{equation}\label{boundv}
\forall (\xi,\theta) \in [-a,a] \times \Theta, \qquad \mu_\theta(\xi,\theta_{\min}) = \mu_\theta(\xi,\theta_{\max}) = 0, \qquad \mu(-a,\theta) = \vert \Theta \vert^{-1}\, , \qquad \mu(a,\theta)  = 0\, .
\end{equation}
we consider the one-parameter family of problems on ${(-a,a)} \times \Theta$:
\begin{equation}\label{eq:tauslab}
 \left\{\begin{array}{l}
-c Z_{\xi}^\tau - g_{\tau}(\theta) Z_{\xi\xi}^\tau - {\alpha} Z_{\theta\theta}^\tau = {r \mu (1 - \nu)}\, , \qquad  (\xi,\theta) \in (-a,a) \times \Theta, \medskip \\
Z_\theta^\tau(\xi,\theta_{\min}) = Z_\theta^\tau(\xi,\theta_{\max}) = 0\, ,\qquad \xi \in (-a,a) ,\medskip\\
Z^\tau(-a,\theta) = \vert \Theta \vert^{-1}, Z^\tau(a,\theta)  = 0, \qquad \theta \in \Theta .
\end{array}
\right.
\end{equation}
We introduce the map 
\begin{equation*}
\mathcal{K}_{\tau}: (c,\mu) \to \left( \epsilon - \nu(0) + c , Z^{\tau} \right), 
\end{equation*}
where $Z_{\tau}$ is the solution of the previous linear system \eqref{eq:tauslab}. The ellipticity of the system \eqref{eq:tauslab} gives that the map $\mathcal{K}_{\tau}$ is a compact map from $\left( X = \mathbb{R} \times \mathcal{C}^{1,{\beta}} \left( (-a,a) \times \Theta \right) , \Vert (c , \mu) \Vert = \max \left( \vert c \vert, \Vert \mu \Vert_{ \mathcal{C}^{1,\beta}} \right) \right)$ onto itself ${(\forall \beta \in (0,1))}$. Moreover, it depends continuously on the parameter $\tau \in \left[ 0 , 1 \right]$. Solving the problem $P_a$ \eqref{eq:slab} is equivalent to proving that the kernel of $\text{Id} - \mathcal{K}_1$ is non-trivial. We can now apply the Leray-Schauder theory. 

We define the open set for $\delta > 0$,
\begin{equation*}
\mathcal{B} = \left \lbrace \; (c,\mu) \; \vert \; 0 < c < c^* + \delta, \; \Vert \mu \Vert_{\mathcal{C}^{1,\beta}\left( (-a,a) \times \Theta \right)} < C_0 + \delta \right\rbrace.
\end{equation*} 
The different a priori estimates of Lemmas \ref{upboundc}, \ref{lem:nc=0}, \ref{bottom}, \ref{lem:nc} give that for all $\tau \in \left[ 0 , 1\right]$ and sufficiently large $a$, the operator $ \text{Id}  - \mathcal{K}_{\tau}$ cannot vanish on the boundary of $\mathcal{B}$. Indeed, if it vanishes on $\partial \mathcal{B}$, there exists a solution $(c,\mu)$ of \eqref{eq:slab} which also satisfies $c \in \left\lbrace 0, c^{*} + \delta \right\rbrace$ or $\Vert \mu \Vert_{\mathcal{C}^{1,\beta}\left( (-a,a) \times \Theta \right)} = C_0 + \delta$ and {$\nu(0) = \eps$}. But this is ruled out by the condition $\eps < \eps_0$. It yields by the homotopy invariance that 
\begin{equation*}
\forall \tau \in \left[ 0 ,1 \right], \quad \text{deg}\left( \text{Id} - \mathcal{K}_{1} , \mathcal{B} , 0 \right) = \text{deg}\left( \text{Id} - \mathcal{K}_{\tau} , \mathcal{B} , 0 \right) = \text{deg}\left(  \text{Id} - \mathcal{K}_{0} , \mathcal{B} , 0 \right).
\end{equation*}
We now need to compute $\text{deg} \left( \text{Id}  - \mathcal{K}_{0} , \mathcal{B} , 0 \right)$. This will be done with two supplementary homotopies. We need these two homotopies to write $\text{Id} - \mathcal{K}_0$ as a tensor of two applications whose degree with respect to $\mathcal{B}$ and $0$ are computable. We first define, {with $\nu_{Z^0}(\cdot) = \int_\Theta Z^0(\cdot,\theta) d\theta$}: 
\begin{equation*}
\mathcal{M}_{\tau}: (c,v) \to  \left( c - (1-\tau)\nu_v(0) - \tau \nu_{ Z^0}(0) + \epsilon , Z^0\right)
\end{equation*}
If there exists $(c,\mu) \in \partial \mathcal{B}$ such that $\mathcal{M}_{\tau}(c,\mu) = (c,\mu)$, then $(c,\mu)$ is such that $Z^0 = \mu$ and $\nu_{Z^0}(0) = \epsilon$. However, such a fixed point $(c,\mu)$ then satisfies 
\begin{equation}\label{eq:vKPP}
\left\{\begin{array}{l}
-c \mu_{\xi} -  \theta_{min} \mu_{\xi\xi} - \mu_{\theta\theta} = r \mu (1 - \nu )\, , \qquad \xi \in (-a,a) \times \Theta,\medskip\\
\mu_\theta(\xi,\theta_{\min}) = \mu_\theta(\xi,\theta_{\max}) = 0\,, \qquad \xi \in (-a,a), \medskip\\
\mu(-a,\theta) = \vert \Theta \vert^{-1}, \quad \mu(a,\theta)  = 0, \qquad \theta \in \Theta,  
\end{array}
\right.
\end{equation}
which is now {closely linked to} the standard Fisher-KPP equation. Indeed, {after integration w.r.t $\theta$}, $\nu$ satisfies 
\begin{equation}\label{eq:tauslab2}
\left\{\begin{array}{l}
-c \nu_{\xi} - \theta_{min} \nu_{\xi\xi} = r \nu (1 - \nu )\,, \qquad \xi \in (-a,a), \medskip\\
 \nu(-a) = 1\, , \quad \nu(a)  = 0\,,
\end{array}
\right.
\end{equation}
{and $\nu(0)=\eps$}. Given a {(unique)} solution $\nu$ of \eqref{eq:tauslab2} after Lemma \ref{propKPP}, we can solve the equation for $v$. {The solution of \eqref{eq:vKPP} is then unique thanks to the maximum principle, and reads $\mu(\xi,\theta) = \frac{\nu(\xi)}{\vert \Theta \vert}$}. As a consequence, such a fixed point cannot belong to $\partial \mathcal{B}$ after all \textit{a priori} estimates of Lemma \ref{propKPP}. Thus, by the homotopy invariance and $ \mathcal{K}_{0} = \mathcal{M}_{0}$, we have
\begin{equation*}
\text{deg} \left( \text{Id}  - \mathcal{K}_{0} , \mathcal{B} , 0 \right) = \text{deg} \left( \text{Id}  - \mathcal{M}_{1} , \mathcal{B} , 0 \right).
\end{equation*}
{The concluding arguments are now the same as in \cite{Berestycki-Nadin}.} Up to the end of the proof, we shall exhibit the dependency of $Z^0$ in $c$: $Z^0 = Z_c$. We now define our last homotopy by the formula
\begin{equation*}
\mathcal{N}_{\tau}: (c,\mu) \to \left( c + \epsilon - \nu_{Z_c}(0), \tau Z_c + (1-\tau) Z_{c_0} \right),
\end{equation*}
where $c_0$ is the unique $c \in \,[0, 2\sqrt{r\theta_{\text{min}}}]$ such that $\nu_{Z_c}(0) = \epsilon$, for $\eps < \eps^* $ and $a(\eps)$ sufficiently large (see again Lemma \ref{propKPP}). If $\mathcal{N}_{\tau}$ has a fixed point, then necessarily $\epsilon = \nu_{Z_c}(0)$ and $\mu = \tau Z_c + (1-\tau) Z_{c_0}$. This gives $\mu = Z_{c_0}$ by uniqueness of the speed $c_0$. Again, such a $\mu$ cannot belong to $\partial \mathcal{B}$ (we recall that {$c_0 < 2 \sqrt{r \theta_{min} } < c^* $ after \eqref{rel6}}). By homotopy invariance and $\mathcal{M}_1 = \mathcal{N}_1$:
\begin{equation*}
\text{deg} \left( \text{Id}  - \mathcal{K}_{1} , \mathcal{B} , 0 \right) = \text{deg} \left( \text{Id}  - \mathcal{K}_{0} , \mathcal{B} , 0 \right) = \text{deg} \left( \text{Id}  - \mathcal{M}_{1} , \mathcal{B} , 0 \right) = \text{deg} \left( \text{Id}  - \mathcal{N}_{0} , \mathcal{B} , 0 \right).
\end{equation*} 
Finally, the operator $\left( \text{Id}  - \mathcal{N}_{0} \right) (c,\mu) = \left( \nu_{Z_c}(0)  - \epsilon , \mu - Z_{c_0} \right)$ is such that $\text{deg} \left( \text{Id}  - \mathcal{N}_{0} , \mathcal{B} , 0 \right) = - 1$. Indeed, the degree of the first component is $-1$ as it is a decreasing function of $c$, and the degree of the second one is $1$.

We conclude that $\text{deg}\left( \text{Id} - \mathcal{K}_{1} , \mathcal{B} , 0 \right) = -1$. Therefore it has a non-trivial kernel whose elements are solution of the slab problem. This proves the Proposition.

\end{proof}

\section{Construction of spatial travelling waves with minimal speed $c^*$.}\label{profileminspeed}

In this Section, we now use the solution of the slab problem \eqref{eq:slab} given by Proposition \ref{slabsol} to construct a wave solution with minimal speed $c^*$. For this purpose, we first pass to the limit in the slab to obtain a profile in the whole space $\R \times \Theta$. Then we prove that this profile necessarily travels with speed $c^*$.

\subsection{Construction of a spatial travelling wave in the full space.}

\begin{lemma}\label{convslab}
Let {$\eps < \min\left( \eps_0 , \eps^* \right)$}. There exists $c_0 \in \left[ 0 , c^* \right]$ such that the system
\begin{equation}\label{convslab2}
\left\{\begin{array}{l}
- c_0 \mu_\xi - \theta \mu_{\xi\xi} - \alpha \mu_{\theta\theta} = r \mu (1 - \nu), \qquad (\xi , \theta ) \in \R \times \Theta,\medskip \\
\mu_\theta(\xi,\theta_{\min}) = \mu_\theta(\xi,\theta_{\max}) = 0, \qquad \xi \in \R, \\
\end{array}
\right.
\end{equation}
has a solution $\mu \in \mathcal{C}_b^2\left( \R \times \Theta \right)$ satisfying $\nu(0) = \eps$. 
\end{lemma}
\begin{proof}[{\bf Proof of Lemma \ref{convslab}}]
{For sufficiently large $a > a_0(\eps)$, Proposition \ref{slabsol} gives a solution $(c^a,\mu^a)$ of \eqref{eq:slab} which satisfies $c^a \in \left[ 0 , c^* \right]$, $\Vert \mu^a \Vert_{L^\infty((-a,a)\times \Theta)} \leq K_0$ and $\nu^a(0) = \eps$. As a consequence, 
\begin{equation*}
\Vert \nu^a \Vert_{L^\infty((-a,a))} \leq \vert \Theta \vert K_0. 
\end{equation*}
The elliptic regularity \cite{Gilbarg} implies that for all $\beta > 0$, $\Vert \mu^a \Vert_{\mathcal{C}^{1,\beta}((-a,a)\times \Theta)} \leq C$ for some $C>0$ uniform in $a$. Then, the Ascoli theorem gives that possibly after passing to a subsequence $a_n \to + \infty$, $(c^a,\mu^a)$ converges towards $(c_0,\mu) \in \left[ 0 , c^* \right] \times \mathcal{C}^{1,\beta}(\R \times \Theta)$ which satisfies \eqref{convslab2} and $\nu(0) = \eps$.}

\end{proof}
\begin{remark}
We do not obtain after the proof that $\sup \nu \leq 1$, and nothing is known about the behaviors at infinity at this stage. Nevertheless, we have an uniform bound $\Vert \nu \Vert_{L^\infty(\R)} \leq \vert \Theta \vert K_0$.
\end{remark}

\subsection{The profile is travelling with the minimal speed $c^*$.}

\begin{lemma}{\bf (Lower bound on the infimum).}\label{inf}
There exists $\delta > 0$ such that any solution $(c,\mu)$ of  
\begin{equation*}
\left\{\begin{array}{l}
- \theta \mu_{\xi\xi} - \alpha \mu_{\theta\theta} - c \mu_\xi = r(1 - \nu)\mu, \qquad \left(\xi , \theta \right) \in \R \times  \Theta, \qquad \medskip\\
\mu_\theta(\xi,\theta_{\min}) = \mu_\theta(\xi,\theta_{\max}) = 0, \qquad \xi \in \R, \qquad \\
\end{array}
\right.
\end{equation*}
with $c \in \left[ 0 , c^*\right]$, $\nu$ bounded and $\inf_{\xi \in \R} \nu(\xi) > 0$ satisfies $\inf_{\xi \in \R} \nu(\xi) > \delta$.
\end{lemma}
\begin{proof}[{\bf Proof of Lemma \ref{inf}}]
We again adapt an argument from \cite{Alfaro} to our context. By the Harnack inequality of Proposition \ref{Harnack}, one has
\begin{equation}\label{Hn}
\forall \left( \xi , \theta , \theta' \right) \in \R \times \Theta^2, \qquad \mu(\xi,\theta) \leq C(\xi) \mu(\xi,\theta'),
\end{equation}
Since \eqref{eqkinwave} is invariant by translation in space, and the renormalization $\nu(0) = \eps$ is not used in the proof of the Harnack inequality, we can take a constant $C(\xi)$ which is independent from $\xi$ \cite{Gilbarg}. This yields

\begin{equation*}
\forall \left( \xi , \theta \right) \in \R \times \Theta, \qquad - \theta \mu_{\xi\xi}(\xi,\theta) - \alpha \mu_{\theta\theta}(\xi,\theta) - c \mu_\xi (\xi,\theta) \geq r(1 - C \Theta \mu(\xi,\theta))\mu(\xi,\theta).
\end{equation*}
{Hence, $\mu$ is a super solution of some elliptic equation with local terms only. For $\eta > 0$ arbitrarily given, we define the family of functions} 
\begin{equation*}
\psi_m(\xi,\theta) = m \left( 1 - \eta \xi^2 \right) Q^*(\theta). 
\end{equation*}
From the uniform $L^\infty$ estimate on $\mu$, there exists $M$ large enough such that $\psi_M(0,\theta) > \mu(0,\theta)$. Moreover, by assumption we have $\psi_m \leq \mu$ for $m= \frac{ \inf_\R \nu}{ C \vert \Theta \vert \Vert Q^* \Vert_\infty} > 0$. As a consequence, we can define 
{\begin{equation*}
m_0:= \sup \lbrace m > 0 , \quad \forall (\xi, \theta) \in \R \times \Theta, \quad \psi_m (\xi,\theta) \leq \mu(\xi,\theta) \rbrace.
\end{equation*}}
As in previous same ideas, see Lemmas \ref{upboundc} and \ref{bottom}, there exists $(x_0 , \theta_0)$ such that $\mu - \psi_{m_0}$ has a zero minimum at this point. We have clearly that $\xi_0 \in \left[ - \frac{1}{\sqrt{\eta}} ; \frac{1}{\sqrt{\eta}} \right]$ since $\psi_m$ is negative elsewhere. We have, at $(\xi_0 , \theta_0)$:
\begin{equation*}
\begin{array}{lcl}
0 & \geq & - \theta_0 \left( \mu - \psi_{m_0} \right)_{\xi\xi} - \alpha \left( \mu - \psi_{m_0} \right)_{\theta \theta} - c \left( \mu - \psi_{m_0} \right)_\xi, \medskip \\
& \geq & r \left(1 - C \vert \Theta \vert \mu \right) \mu + \theta_0 \left( \psi_{m_0} \right)_{\xi\xi} + \alpha \left( \psi_{m_0} \right)_{\theta \theta} + c \left(\psi_{m_0}\right)_\xi, \medskip \\
& \geq & r \left(1 - C \vert \Theta \vert \mu \right) \mu - 2 \eta m_0 \theta_0 Q^*(\theta_0) -  \left( -\lambda c( \lambda ) + \theta_0 \lambda^2 + r \right) \psi_{m_0}(\xi_0, \theta_0) - 2 c\eta \xi_0 m_0 Q^*(\theta_0),  \medskip\\
& \geq & \mu(\xi_0,\theta_0) \left( \lambda^* c^* - \theta_0 (\lambda^*)^2 - r C \vert \Theta \vert \mu(\xi_0,\theta_0) \right)  - 2 m_0 Q^*(\theta_0) \left( \eta \theta_0 + \eta \xi_0 c \right).
\end{array}
\end{equation*}
It follows from $\mu(\xi_0,\theta_0) \geq \frac{\nu(\xi_0)}{C \vert \Theta \vert}$ \eqref{Hn}, from the inequalities $\vert \xi_0 \vert \leq \frac{1}{\sqrt{\eta}}$, $c \leq c^*$, $m_0 \leq M$ and  the fact that for all $\theta_0 \in \Theta$, the quantity $c^* - \theta_0 \lambda^* - \theta_{\text{min}} \lambda^*$ is positive (see \eqref{rel4}) that
\begin{equation*}
\begin{array}{lcl}
\mu(\xi_0 , \theta_0) &\geq& \displaystyle  \frac{\lambda^* \left( c^* - \theta_0 \lambda^* \right)}{r C \vert \Theta \vert } - \frac{2 C \Theta M \Vert Q^* \Vert_{\infty}\left(\eta \theta_{\text{max}} +  \sqrt{\eta} c^* \right)}{r C \vert \Theta \vert \nu(\xi_0)}, \medskip \\
&\geq& \displaystyle \frac{\theta_{\text{min}} \left( \lambda^* \right)^2 }{r C \vert \Theta \vert} - \frac{2 C \Theta M \Vert Q^* \Vert_{\infty}\left( \sqrt{\eta} c^* + \eta \theta_{\text{max}} \right)}{r C \vert \Theta \vert \,{ (\inf_{\xi \in \R} \nu)} }.\\
\end{array}
\end{equation*}
{Recalling $\inf_{\xi \in \R} \nu > 0$ and taking arbitrarily small values} of $\eta > 0$, we have necessarily $\mu(\xi_0, \theta_0) \geq \frac{\theta_{min} \left( \lambda^* \right)^2 }{{2} C r \vert \Theta \vert}$. Since $\mu$ and $\psi_{m_0}$ coincide at $(\xi_0,\theta_0)$, we have $m_0 \geq \frac{\theta_{min} \left( \lambda^* \right)^2 }{{2} r C \vert \Theta \vert \Vert Q^* \Vert_{\infty} }$. The definition of $m_0$ now gives 
\begin{equation*}
\forall (\xi,\theta) \in \R \times \Theta, \qquad \mu(\xi,\theta) \geq \frac{\theta_{min} \left( \lambda^* \right)^2 }{{2} C \vert \Theta \vert r \Vert Q^* \Vert_{\infty} } \left( 1 - \eta \xi^2 \right) Q^*(\theta).
\end{equation*}
Since $\eta$ is arbitrarily small, we have necessarily $\nu(\xi) \geq \delta:= \frac{\theta_{min} \left( \lambda^* \right)^2 }{{2} C \vert \Theta \vert r \Vert Q^* \Vert_{\infty} }$ for all $\xi \in \R$.

\end{proof}

We deduce from this Lemma that up to choosing $\eps < \delta$, the solution necessarily satisfies $\inf_{\R} \nu(\xi) = 0$. {Since this infimum cannot be attained, we have necessarily $\liminf_{\xi \to + \infty} \nu(\xi) = 0$ (up to $\xi \to - \xi$ and $c \to -c$)}. We now prove that this enforces $c = c^*$ for our wave. For this purpose, we show that a solution going slower than $c^*$ cannot satisfy the $\liminf$ condition by a sliding argument.

\begin{proposition}\label{prop:minspeed}
Any solution $(c,\mu)$ of the system 
\begin{equation}\label{eq:minspeed}
\left\{\begin{array}{l}
- \theta \mu_{\xi\xi} - \alpha \mu_{\theta\theta}  - c \mu_\xi = r \mu (1 - \nu), \qquad (\xi , \theta ) \in \R \times \Theta, \qquad  \medskip\\
\mu_\theta(\xi,\theta_{\min}) = \mu_\theta(\xi,\theta_{\max}) = 0, \qquad \xi \in \R, \qquad \\
\end{array}
\right.
\end{equation}
with $c\geq0$ and $\liminf_{\xi \in\R} \nu(\xi) = 0$ satisfies necessarily $c \geq c^*$. 
\end{proposition}

As a consequence, the solution given after Lemma \ref{convslab} goes with the speed $c^*$. This latter speed appears to be the minimal speed of existence of nonnegative travelling waves, similarly as for the Fisher KPP equation.

\begin{proof}[{\bf Proof of Proposition \ref{prop:minspeed}}]
We again play with subsolutions. {By analogy with the Fisher-KPP equation, we shall use oscillating fronts associated with speed $c < c^*$ to "push" solutions of \eqref{eq:minspeed} up to the speed $c^*$}. We now proceed like in \cite{Bouin-Calvez-Nadin}.

Let us now consider the spectral problem for complex values of $\lambda$:
\begin{equation}\label{eq:eigenpbs}
\begin{cases} 
\alpha \partial^2_{\theta\theta} Q_\lambda(\theta) + \left( - \lambda c + \theta \lambda^2 + r - s \right) Q_\lambda(\theta) = 0\, ,\medskip\\
\partial_\theta Q_\lambda(\theta_{min}) = \partial_\theta Q_\lambda(\theta_{max}) =0\,.
\end{cases}
\end{equation}
When $s = 0$ we know from Proposition \ref{propspec} that for $c = c^*$ there exists some \textit{real} $\lambda^* > 0$ such that the spectral problem is solvable with a positive eigenvector. Moreover, the minimal speed is increasing with respect to $r$. Indeed, for all $r < s$ and $\lambda > 0$, one has
\begin{equation*}
\lambda c_r(\lambda) = r +  \lambda^2 \theta_{\text{max}} - \gamma( \lambda ) < s +  \lambda^2 \theta_{\text{max}} - \gamma( \lambda ) = \lambda c_s(\lambda)
\end{equation*}
and thus $c_r^* < c_s^*$.

Now suppose by contradiction that $c < c^*$. Take $c < \bar c < c^*$, $s>0$. One can choose $s( \bar c ) > 0 $ such that $\bar c$ is the minimal speed of the spectral problem \eqref{eq:eigenpbs}. There exists $\lambda_c:= \lambda_R + i \lambda_I \in \C$ with $\text{Re}(\lambda_c) > 0$ such that there exists $Q_{\lambda_c}: \Theta \mapsto \C$ which solves the spectral problem. A continuity argument ensures that $\text{Re} \left(Q_{\lambda_c}\right) > 0$ since $\text{Re} \left(Q_{\lambda_{\bar c}}\right) > 0$ when $\bar c$ is sufficiently close to $c$. 

Let us now define the real function
\begin{equation*}
\psi(\xi,\theta):= \text{Re} \left( e^{- \lambda_c \xi} Q_{\lambda_c} \left( \theta \right)\right) = e^{-\lambda_{R} \xi}\left[ \text{Re} \left(Q_{\lambda_c}(\theta)\right) \cos (\lambda_{I} \xi) + \text{Im} \left(Q_{\lambda_c}(\theta)\right) \sin (\lambda_{I} \xi) \right].
\end{equation*} 
For all $\theta \in \Theta$, one has $\psi\left( 0 , \theta \right) > 0$ and $\psi\left( \pm \frac{\pi}{\lambda_I} , \theta \right) < 0$. As a consequence, there exists an open subdomain $\mathcal{D} \subset \Omega:= \left[ - \frac{\pi}{\lambda_I} , \frac{\pi}{\lambda_I} \right] \times \Theta$ such that $\psi > 0$ on $\mathcal{D}$ and $\psi$ vanishes on the boundary $\partial \mathcal{D}$. From the Harnack estimate of Proposition \ref{Harnack}, there exists a constant $C$ which depends on $\vert \mathcal{D} \vert$ such that one has for all $\xi \in \R$,
\begin{equation*}
\forall (z,\theta,\theta') \in \mathcal{D} \times \Theta, \qquad u(z + \xi, \theta) \leq C u(\xi,\theta')
\end{equation*}
By construction, one has
\begin{equation*}
- \theta \psi_{\xi\xi} - \alpha \psi_{\theta \theta} - c \psi_\xi - r \psi = - s(\bar c) \psi.
\end{equation*}
Thus, for all $m \geq 0$, the function $v:= \mu - m\psi$ satisfies
\begin{equation*}
- \theta v_{\xi\xi} - \alpha v_{\theta \theta} - c v_\xi - r v = m s( \bar c)\psi - r \nu(\xi) \mu
\end{equation*}
There now exists $m_0$ such that $v$ attains a zero minimum at $(\xi_0,\theta_0) \in \mathcal{D}$. One deduces $\nu(\xi_0) \geq \frac{s(\bar c)}{r}$. We conclude by the Harnack estimate that $\nu(0) \geq \frac{s(\bar c)}{rC}$. 

{We now want to translate the argument in space. For this purpose, we define, for $\zeta \in \R$, the function $h(\xi,\theta):= \mu(\xi + \zeta,\theta)$. It also satisfies \eqref{eq:minspeed}. As a consequence, for all $\zeta \in \R$, $\nu\left( \zeta \right) = \int_{\Theta} h(0,\theta) d \theta \geq \frac{s(\bar c)}{rC}$. We emphasize that the renormalization $\nu(0) = \eps$, which is the only reason for which \eqref{eq:slab} is not invariant by translation, is not used here. We then obtain $\inf_{\xi\in\R} \nu(\xi) \geq \frac{s(\bar c)}{rC}$. This contradicts the property $\liminf_{\xi\in\R} \nu(\xi) = 0$.}

\end{proof}

\subsection{The profile has the required limits at infinity.}

\begin{proposition}\label{limits}
Any solution $(c,\mu)$ of the system 
\begin{equation*}
\left\{\begin{array}{l}
- \theta \mu_{\xi\xi} - \alpha \mu_{\theta\theta}  - c \mu_\xi = r \mu (1 - \nu), \qquad (\xi , \theta ) \in \R \times \Theta, \qquad  \medskip\\
\partial_\theta \mu(\xi,\theta_{\min}) = \partial_\theta  \mu(\xi,\theta_{\max}) = 0, \qquad \xi \in \R, \qquad \\
\end{array}
\right.
\end{equation*}
with $c\geq0$ and $\nu(0) = \eps$ satisfies
\begin{enumerate}
\item There exists $m>0$ such that $\forall \xi \in ]-\infty , 0], \quad \mu\left( \xi ,\cdot \right) > m Q(\cdot)$,
\item $\lim_{\xi \to +\infty} \mu(\xi,\cdot) =0.$
\end{enumerate}
\end{proposition}
\begin{proof}[{\bf Proof of Proposition \ref{limits}}]
We again adapt to our case an argument from \cite{Alfaro}. By the Harnack inequality of Proposition \ref{Harnack}, there exists $\widetilde C$ such that one has
{\begin{equation}\label{HHn}
\inf_{\xi \in [-1,0] \times \Theta} \mu(\xi,\theta) \geq \frac{\eps}{\widetilde{C} \vert \Theta \vert },
\end{equation}}
recalling $\nu(0) = \eps$. Also recalling 
\begin{equation*}
\forall \left( \xi , \theta , \theta' \right) \in \R \times \Theta^2, \qquad \mu(\xi,\theta) \leq C \mu(\xi,\theta'),
\end{equation*}
we obtain
\begin{equation*}
\forall \left( \xi , \theta \right) \in \R \times \Theta, \qquad - \theta \mu_{\xi\xi}(\xi,\theta) - \alpha \mu_{\theta\theta}(\xi,\theta) - c \mu_\xi (\xi,\theta) \geq r(1 - C \vert \Theta \vert \mu(\xi,\theta))\mu(\xi,\theta).
\end{equation*}
Let us define, for $m = \frac12 \min \left( \frac{\eps}{ \vert \Theta \vert \widetilde{C} \Vert Q^* \Vert_{\infty} } , \frac{\theta_{min} (\lambda^*)^2}{rC \Vert Q^* \Vert_{\infty} \vert \Theta \vert} \right)$ and $\eta > 0$ arbitrarily given, the function 
\begin{equation*}
\psi_\eta(\xi,\theta) = m \left( 1 + \eta \xi \right) Q^*(\theta). 
\end{equation*}
on $\left] - \infty , 0\right] \times \Theta$.
We have, 
\begin{equation*}
\forall (\xi,\theta) \in \left]-\infty,-1\right] \times \Theta, \qquad \psi_1(\xi,\theta) = m \left( 1 + \xi \right) Q^*(\theta) \leq 0 \leq \mu(\xi,\theta). 
\end{equation*}
Moreover, for $(\xi,\theta) \in ]-1 , 0] \times \Theta$, using \eqref{HHn}, we have 
\begin{equation*}
\psi_1(\xi,\theta) = m \left( 1 + \xi \right) Q^*(\theta) \leq m \Vert Q^* \Vert_{\infty} \leq \frac12 \frac{\eps \Vert Q^* \Vert_{\infty}}{ \vert \Theta \vert \widetilde{C} \Vert Q^* \Vert_{\infty} }  \leq  \inf_{\xi \in [-1,0] \times \Theta} \mu(\xi,\theta) \leq \mu(\xi,\theta).
\end{equation*}
As a consequence we can define 
\begin{equation*}
\eta_0:= \min \lbrace \eta > 0 , \forall (\xi,\theta) \in \left]- \infty , 0\right] \times \Theta,  \psi_\eta (\xi,\theta) \leq \mu(\xi,\theta) \rbrace \in [0,1].
\end{equation*}

We will now prove that $\eta_0 = 0$ by contradiction. Suppose that $\eta_0 > 0$. {We apply the same technique as in the proofs of Lemmas \ref{upboundc} and \ref{bottom}: there exists $(\xi_0 , \theta_0)$ such that $\mu - \psi_{\eta_0}$ has a zero minimum at this point. Moreover, we have }$\xi_0 \in \left[ - \frac{1}{\eta_0} ; 0 \right]$ since $\psi_\eta$ is negative elsewhere. Moreover, $\xi_0$ cannot be $0$ since this would give $\mu(0,\theta_0) = m Q^*(\theta_0) \leq \frac12 \frac{\eps}{ \vert \Theta \vert \widetilde{C} } $ and this would contradict \eqref{HHn}. We have, at $(\xi_0 , \theta_0)$:
\begin{align*}
0 & \geq  - \theta \left( \mu - \psi_{\eta_0} \right)_{\xi\xi} - \alpha \left( \mu - \psi_{\eta_0} \right)_{\theta \theta} - c \left( \mu - \psi_{m_0} \right)_\xi \medskip \\
& \geq  r \left(1 - C \Theta \mu \right) \mu + \theta \left( \psi_{\eta_0} \right)_{\xi\xi} + \alpha \left( \psi_{\eta_0} \right)_{\theta \theta} + c \left(\psi_{m_0}\right)_\xi \medskip \\
& \geq  r \left(1 - C \Theta \mu \right) \mu - \psi_{\eta_0}(\xi_0, \theta_0) \left( -\lambda^* c^* + \theta_0 \left(\lambda^*\right)^2 + r \right) + c m_0 \eta  Q^*(\theta_0)  \medskip\\
& \geq  \mu(\xi_0,\theta_0) \left( \lambda^* c^* - \theta_0 (\lambda^*)^2 - r C \vert \Theta \vert \mu(\xi_0,\theta_0) \right)   + c m_0 \eta  Q^*(\theta_0) \\
& \geq  \mu(\xi_0,\theta_0) \left( \lambda^* c^* - \theta_0 (\lambda^*)^2 - r C \vert \Theta \vert \mu(\xi_0,\theta_0) \right)  
\end{align*}
It yields
\begin{equation*}
\frac{\theta_{min} (\lambda^*)^2}{rC \vert \Theta \vert} \leq \mu(\xi_0,\theta_0) = \psi_{\eta_0}(\xi_0,\theta_0) \leq m \Vert Q^* \Vert_{\infty}.
\end{equation*}
and this contradicts the very definition of $m$. As a consequence, $\eta_0 = 0$ and 
\begin{equation*}
\forall (\xi,\theta) \in \R^- \times \Theta, \qquad \mu(\xi,\theta) \geq m Q^*(\theta)
\end{equation*}
In particular, $\inf_{\R^-} \nu \geq m$ holds.

We now prove that $\lim_{\xi \to +\infty} \mu(\xi,\cdot) =0$. It is sufficient to prove that $\lim_{\xi \to \infty}\nu(\xi) = 0$. Suppose that there exists $\delta$ a subsequence $\xi_n \to + \infty$ such that $\forall n \in \N, \;\nu(\xi_n)\geq \delta$. Adapting the preceding proof we obtain that for all $n \in \N$,
\begin{equation}\label{last}
\forall (\xi,\theta) \in \left] -\infty , \xi_n \right] \times \Theta, \qquad \nu(\xi) \geq \frac12 \min \left( \frac{\delta}{ \vert \Theta \vert \widetilde{C} \Vert Q^* \Vert_{\infty} } , \frac{\theta_{min} (\lambda^*)^2}{rC \Vert Q^* \Vert_{\infty} \vert \Theta \vert} \right).
\end{equation}
Hence \eqref{last} is true for all $\xi \in \R$ and Lemma \ref{inf} gives the contradiction since the normalization $\eps$ is well chosen.  

\end{proof}

\section*{Acknowledgments}

The authors are extremely grateful to Sepideh Mirrahimi for very fruitful comments and earlier computations on this problem. The authors also thank Olivier Druet for the proof of Proposition \ref{Harnack} and Léo Girardin for valuable suggestions.

\section*{Appendix A: A Harnack inequality up to the boundary.}

We emphasize here a useful Harnack inequality for \eqref{eqkinwave} which is true up to the boundary in the direction $\theta$. This is possible thanks to the  Neumann boundary conditions in this direction. 

\begin{proposition}\label{Harnack}
Suppose that $\mu$ is a solution of \eqref{eqkinwave} such that the total density $\nu$ is locally bounded. Then for all $0 < b < + \infty$, there exists a constant $C(b) < + \infty$ such that the following Harnack inequality holds:
\begin{equation*}
\forall (\xi,\theta,\theta') \in (-b,b) \times \Theta \times \Theta, \qquad \mu(\xi,\theta) \leq C(b) \mu(\xi,\theta').
\end{equation*}
\end{proposition}

\begin{proof}[{\bf Proof of Proposition \ref{Harnack}}]

{One has to figure out how to obtain the validity of the Harnack inequality up to the boundary in $\Theta$. Indeed, it holds on sub-compacts sets thanks to the standard elliptic regularity, given that the density $\nu$ is bounded. To obtain the full Harnack estimate, we consider the equation \eqref{eqkinwave} after a reflection with respect to $\theta = \theta_{min}$ and $\theta = \theta_{max}$ and for positive values of $\theta$. One obtains the following equation in the weak sense
\begin{equation*}
\forall (\xi, \theta) \in \R \times \left( \R^{+*} \cap \left( \R \backslash \lbrace \theta_{min} + \Theta \Z \rbrace \right)\right) , \qquad - c \mu_\xi (\xi,\theta) - g(\theta) \mu_{\xi \xi}(\xi,\theta) - \alpha \mu_{\theta\theta}(\xi,\theta) = r \mu(\xi,\theta) (1 - \nu(t,\xi))\,.
\end{equation*}
The crucial point is that this equation is also satisfied on the boundaries $\theta = \R^+ \cap \lbrace \theta_{min} + \Theta \Z \rbrace $ thanks to the Neumann boundary conditions. Indeed, no Dirac mass in $\theta = \R^+ \cap \lbrace \theta_{min} + \Theta \Z \rbrace$ arises while computing the second derivative $\partial_{\theta\theta}$ in the symmetrized equation.}

\end{proof}

\section*{Appendix B: Proof of the interpolation estimate.}

We prove here the interpolation estimate which is needed in {\bf \# Step 2} of the proof of Lemma \ref{lem:nc}. Since $\theta$ is the only variable playing a role here, we denote $g(\theta) = n(t,x,\theta)$. Let $(\theta,\theta') \in \Theta^2$. For technical reason, we impose $|\theta-\theta'|^{-1}\geq e^4$. We set $K = |\theta-\theta'|^{-1}$. We first prove a H\"older-like estimate,
\begin{equation*}
\vert g(\theta) - g(\theta') \vert \leq \frac12 \|g\|_{H^{3/2}}|\theta - \theta'|\log\left(|\theta - \theta'|^{-1}\right)\, .
\end{equation*}
For this purpose, we use Fourier expansions. We recall the definition of the fractional Sobolev norm 
\begin{equation*}
\Vert f \Vert_{H^{\frac32}} = \left( \sum_{k\in\Z^*} |k|^3  |\hat f(k)|^2 \right)^{1/2},
\end{equation*}
where $\hat f$ is the Fourier transformation of $f$. We then have
\begin{align*}
|g(\theta) - g(\theta')| &  = \sum_{|k|\leq K} |\hat g(k)|\left|  e^{ik\theta} - e^{ik\theta'}\right| +  \sum_{|k|> K} |\hat g(k)|\left|  e^{ik\theta} - e^{ik\theta'}\right|  \\
& \leq \sum_{|k|\leq K}   |k| |\hat g(k)| |\theta - \theta'| + 2 \sum_{|k|> K}|\hat g(k)|\\  
& \leq |\theta - \theta'| \sum_{|k|\leq K}  |k|^{3/2}  |\hat g(k)| |k|^{-1/2} + 2 \sum_{|k|> K} |k|^{3/2}|\hat g(k)||k|^{-3/2} \\
& \leq |\theta - \theta'| \left( \sum_{k\in\Z^*} |k|^3  |\hat g(k)|^2 \right)^{1/2} \left( \sum_{|k|\leq K} |k|^{-1}   \right)^{1/2} + 2 \left( \sum_{k\in\Z^*} |k|^3  |\hat g(k)|^2 \right)^{1/2} \left( \sum_{|k|> K} |k|^{-3}   \right)^{1/2}\\
& \leq \|g\|_{H^{3/2}} \left(  |\theta - \theta'|\log\left(|\theta - \theta'|^{-1}\right) + 2  |\theta - \theta'| \right) \\
& \leq \|g\|_{H^{3/2}} |\theta - \theta'| \left( 2 + \log\left(|\theta - \theta'|^{-1}\right) \right) \\
& \leq \frac12 \|g\|_{H^{3/2}}|\theta - \theta'|\log\left(|\theta - \theta'|^{-1}\right)\, .
\end{align*}
Next we estimate
\begin{equation*}
g(\theta) = g(\theta) - g(\theta') + g(\theta')  \leq \frac12 \|g\|_{H^{3/2}}|\theta - \theta'|\log\left(|\theta - \theta'|^{-1}\right) + g(\theta').
\end{equation*}
We integrate for $|\theta - \theta'|\leq \delta/2$, and divide by $\delta$ where $\delta \leq e^{-4}$. 
\begin{equation*} 
g(\theta) \leq   \frac12 \|g\|_{H^{3/2}} \delta \log \delta^{-1} +  \dfrac{\|g\|_{L^1}}\delta\, . \end{equation*}
Choosing  $\delta = \min\left( e^{-4}, \left( \|g\|_{L^1}/\|g\|_{H^{3/2}}\right)^{1/2}\right)$, we get eventually
\[ 
\begin{cases}
\|g\|_{L^\infty} \leq \dfrac1{2} \left(\|g\|_{L^1} \|g\|_{H^{3/2}}\right)^{1/2}\left( \dfrac12 \log\left(\dfrac{\|g\|_{H^{3/2}}}{\|g\|_{L^1}}\right) + 2\right)  & \mathrm{if}\quad \dfrac{\|g\|_{L^1}}{\|g\|_{H^{3/2}}} \leq e^{-8} \medskip \\
\|g\|_{L^\infty} \leq 3 e^4 \|g\|_{L ^1} & \mathrm{otherwise} 
\end{cases} \]
In order to simplify the forthcoming computations, we use the simple estimate $(\forall \delta<e^{-4})\; \log \delta^{-1} + 2 \leq C \delta^{-1/3}$ for some constant $C$. We obtain finally 
\[
\begin{cases}
\|g\|_{L^\infty}^3 \leq C \|g\|_{L^1} \|g\|_{H^{3/2}}^2 & \mathrm{if}\quad \dfrac{\|g\|_{L^1}}{\|g\|_{H^{3/2}}} \leq \dfrac1C \medskip, \\
\|g\|_{L^\infty} \leq C \|g\|_{L ^1} & \mathrm{otherwise}. 
\end{cases} \]

\end{document}